\documentclass[112pt]{amsart}
\usepackage{mathrsfs}
\usepackage{amsfonts}
\usepackage{amssymb}
\usepackage[papersize={6.7in,9.4in},textwidth=13.4cm,textheight=20cm,centering]{geometry}
\usepackage{enumerate}

    \usepackage[utf8]{inputenc}
    \usepackage[T1]{fontenc}
    \usepackage{libertine} 
    \usepackage[libertine]{newtxmath}

\usepackage{amsmath}
\numberwithin{equation}{section}

\usepackage{multirow}

\usepackage[colorlinks=true,citecolor=blue,linkcolor=blue]{hyperref}
\hypersetup{
pdfstartpage=1,
pdfstartview=FitH}

\usepackage{amsthm,amssymb}
\usepackage{indentfirst}
\usepackage[leqno]{amsmath}

\allowdisplaybreaks

\newtheorem*{theorem*}{Theorem}
\newtheorem{theorem}{Theorem}
\newtheorem{lemma}{Lemma}

\newtheorem{corollary}{Corollary}

\newtheorem*{claim*}{Claim}

\theoremstyle{definition}
\newtheorem{remark}{Remark}

\DeclareMathOperator{\Mod}{mod}

\renewcommand{\bmod}[1]{\,(\Mod{ #1})}

\newcommand{\bN}{\mathbf{N}}

\newcommand{\bZ}{\mathbf{Z}}

\newcommand{\cE}{\mathcal{E}}

\newcommand{\cL}{\mathcal{L}}

\newcommand{\cP}{\mathcal{P}}

\def\le{\leqslant}
\def\leq{\leqslant}

\begin{document}

\title{A problem of D. H. Lehmer in short intervals. I}
\author{Qixiang Shen}

\address{School of Mathematics and Statistics, Xi'an Jiaotong University, Xi'an 710049, P. R. China}
\email{qxshen@yeah.net}

\begin{abstract}
     A problem of D. H. Lehmer suggests to study the number of integers, each of which has different parity from its multiplicative inverse modulo $q.$
     We obtain an asymptotic formula for the number of such integers in very short intervals as long as $q$ is squarefree and has good factorizations, using arithmetic exponent pairs to estimate incomplete Kloosterman sums. This improves an early result of Z. Zheng in the above special cases. We also prove a prime-variable version with the aid of an estimate for incomplete Kloosterman sums in prime variables, due to Bourgain.
\end{abstract}


\maketitle


\section{Introduction and our main result}
Let $q>1$ be an odd integer. For any $a\in\bZ$ with $(a, q)=1$, let $\bar{a}$ denote the unique integer satisfying $1\le \bar{a}\le q-1$ and $a\bar{a}\equiv 1\bmod q$. Define
\begin{align*}
    \mathcal{L}(q)=\{1\le a\le q-1:(a, q)=1, 2\nmid a+\bar{a}\},
    \end{align*}
and
\begin{align*}
    L(q)=\vert\mathcal{L}(q)\vert.
\end{align*}
The problem of determining the asymptotic behavior of $L(q)$ was first proposed by D. H. Lehmer and recorded as Problem F12 in Guy’s book \cite{Gu81}. When $p$ is an odd prime, a simple observation shows that $L(p)\equiv 2\bmod 4$ if $p\equiv 1\bmod 4$ and $L(p)\equiv 0\bmod 4$ if $p\equiv 3\bmod 4$. 

This above problem was initiated by Zhang \cite{Zha92}, who used Dirichlet $L$-functions and estimates for character sums and Kloosterman sums to prove an asymptotic formula for $L(p)$ with odd primes $p$. In general, for each odd integer $q>1$, he can prove the following asymptotic formula
\begin{align}\label{eq:zhang}
    L(q)=\frac{1}{2}\varphi(q)+O(q^{1/2}\tau(q)^2\log^2q).
\end{align}
See \cite{Zha93,Zha94,LRS07,LY09,Sh09} for other relevant progresses.

Zheng \cite{Zhe93} generalized Zhang's result to short intervals. We put
\begin{align*}
     \cL(N,q)=[1,N]\cap \cL(q),\ \ \ L(N, q)=\vert\cL(N ,q)\vert.
\end{align*}
For each odd integer $q>1$ and $1\le N\le q-1$, Zheng \cite{Zhe93} proved that
\begin{align}\label{eq:zheng}
    L(N,q)=\frac{1}{2}N\varphi(q)q^{-1}+O(q^{1/2}\tau(q)\log^2q).
\end{align}
Taking $N=q-1$ in (\ref{eq:zheng}), one obtains the asymptotic formula in \eqref{eq:zhang}. One sees that (\ref{eq:zheng}) is nontrivial for all $N\gg q^{1/2+\varepsilon}$ with $\varepsilon>0$. 

In fact, the original counting problem can ultimately be reduced to estimates for exponential sums. This reduction is formulated in Lemma \ref{lm:counting-expsums} in a weighted form. Consequently, an asymptotic formula for $L(N,q)$ follows once nontrivial bounds for these sums are available in very short intervals. The main aim of this paper is to use new bounds for incomplete Kloosterman sums, due to Wu and Xi \cite{WX21}, to improve (\ref{eq:zheng}) when $q$ is squarefree and has sufficiently good
factorizations.

\begin{theorem}\label{thm:main}
    Let $1\le N\le q-1$. Suppose that $q$ is squarefree and has only prime factors not exceeding $q^{\eta}$ with $\eta>0$ sufficiently small. For all $m\in\bZ^+$ and $\varepsilon>0,$ we have
        \begin{align*}
    L(N, q)=\frac{1}{2}N\varphi(q)q^{-1}+O(N^{\lambda-\kappa+O(\eta)}q^{\kappa+\varepsilon}),
\end{align*}
where
\begin{align*}
    (\kappa, \lambda)=\Big(\frac{1}{2^{m+2}-2},\frac{2^{m+2}-m-3}{2^{m+2}-2}\Big).
\end{align*}
The implied constant depends at most on $m,\eta$ and $\varepsilon$.
\end{theorem}

By taking $m$ sufficiently large and $\varepsilon$ sufficiently small, Theorem \ref{thm:main} gives an asymptotic formula for all $N\gg q^{\varepsilon}$, provided that $\eta$ is sufficiently small in terms of $\varepsilon$. This allows us to derive the following immediate consequence.

\begin{corollary}\label{coro}
Under the hypotheses of Theorem \ref{thm:main}, for any $\varepsilon>0$, there exists an integer $a$ with $1\le a\ll q^{\varepsilon}$ such that $(a, q)=1$ and $2\nmid a+\bar{a}$.
\end{corollary}

Finally, we also mention that, as one may see from the subsequent proof, one may replace the interval $[1,N]$ by $(M,M+N]$ in Theorem \ref{thm:main}, for any given $M$. On the other hand, the statements in Theorem \ref{thm:main} and Corollary \ref{coro} would fail if $q$ is prime, in which case we revisit this problem in \cite{Sh26}.

\section{Preliminary lemmas}

We need the following bound for incomplete Kloosterman sums, which follows from the theory of arithmetic exponent pairs in \cite{WX21}.
\begin{lemma}\label{lm:incompletekloosterman}
     Suppose that $q$ is squarefree and has only prime factors not exceeding $q^{\eta}$ with $\eta>0$ sufficiently small. For all $h\in \bZ$, any interval $I$ and $m\in\bN$, we have
     \begin{align*}
         \sum_{\substack{n\in I\\(n, q)=1}}e\Big(\frac{h\bar{n}}{q}\Big)\ll \frac{\vert I\vert}{q}(h, q)+q^{\kappa+\varepsilon}\vert I\vert^{\lambda-\kappa+O(\eta)}(h, q),
     \end{align*}
where
\begin{align}\label{eq:exponent pairs}
    (\kappa, \lambda)=\Big(\frac{1}{2^{m+2}-2},\frac{2^{m+2}-m-3}{2^{m+2}-2}\Big).
\end{align}
The implied constant depends at most on $m,\eta$ and $\varepsilon$.
\end{lemma}
\begin{proof}
    See \cite[Section 9]{WX21} or \cite[Lemma 3.7]{XZ26}.
\end{proof}

\begin{remark}
Suppose that $(h,q)=1.$ One sees Lemma \ref{lm:incompletekloosterman} beats the trivial bound for all $|I|>q^{\varepsilon}$ if $\eta$ is small enough and $m$ is large enough. However, for a general $q$, one can only take $m=0$, producing a non-trivial bound for 
$|I|>q^{1/2+\varepsilon}$, which is essentially the argument in \cite{Zhe93}.
\end{remark}

\begin{lemma}\label{lm:counting-expsums}
Let $q>1$ be an odd integer, $\{{a_n}\}$ an arbitrary coefficient. Then we have
    \begin{align}\label{eq:transform into character sums}
        \sum_{n\in \mathcal{L}(N, q)}a_n=\frac{1}{2}\sum_{\substack{n\le N\\(n, q)=1}}a_n+O\Big(\sum_{j=1}^2\sum_{0<\vert r\vert<q/2}\frac{1}{|r|}\Big\vert\sum_{n\in I_j}a_{n,j}e\Big(\frac{ k_jr\overline{n}}{q}\Big)\Big\vert\Big),
        \end{align}
where $I_1=[1, N/2]$, $a_{n, 1}=a_{2n}$, $4k_1\equiv \pm1\bmod q$ and $I_2=[1, N]$, $a_{n, 2}=a_{n}$, $2k_2\equiv \pm1\bmod q$. The implied $O$-constant is absolute.
\end{lemma}
\begin{proof}
    By the definition of $\mathcal{L}(q)$, we have
    \begin{align}
        \sum_{n\in \cL (N, q)}a_n&=\frac{1}{2}\mathop{\sum_{n\le N}\sum_{b<q}}_{nb\equiv 1\bmod q}a_n(1-(-1)^{n+b})\notag\\
        &=\frac{1}{2}\sum_{\substack{n\le N\\(n, q)=1}}a_n-\frac{1}{2}\mathop{\sum_{n\le N}\sum_{b<q}}_{nb\equiv 1\bmod q}a_n(-1)^{n+b}\notag\\
        &=\frac{1}{2}\sum_{\substack{n\le N\\(n, q)=1}}a_n-\frac{1}{2\varphi(q)}\sum_{\chi\bmod q}\Big(\sum_{n\le N}(-1)^na_n\chi(n)\Big)\Big(\sum_{1\le b\le q-1}(-1)^b\chi(b)\Big).\label{eq:restrict chi}
    \end{align}
For the sum over $n$ in (\ref{eq:restrict chi}), we may write
\begin{align*}
    \sum_{n\le N}(-1)^na_n\chi(n)=
    \sum_{n\le N}(1+(-1)^n)a_n\chi(n)-\sum_{n\le N}a_n\chi(n)=2\sum_{n\le N/2}a_{2n}\chi(2n)-\sum_{n\le N}a_n\chi(n).
\end{align*}
For the sum over $b$, we have
\begin{align*}
    \sum_{b=1}^{q-1}(-1)^b\chi(b)=
    \begin{cases}
        0, & \text{if} \  \chi(-1)=1,\\
        2\chi(2)\sum_{1\le b\le q/2}\chi(b), & \text{if} \ \chi(-1)=-1,
    \end{cases}
\end{align*}
hence only odd characters contribute to (\ref{eq:restrict chi}). 
Thus the second term in (\ref{eq:restrict chi}) equals
\begin{align}\label{eq:second term}
    \frac{1}{2\varphi(q)}\sum_{\chi(-1)=-1}\chi(2)\Big(2\sum_{n\le N/2}\chi(2n)-\sum_{n\le N}\chi(n)\Big)\Big(\sum_{1\le b\le \frac{q-1}{2}}\chi(b)\Big).
\end{align}

 For each Dirichlet character $\chi\bmod q$, define the Gauss sum $\tau(r, \chi)$ by
\begin{align}\label{eq:Gauss sum}
    \tau(r, \chi)=\sum_{c\bmod q}\chi(c)e\Big(\frac{rc}{q}\Big).
\end{align}
For any non-principal character $\chi\bmod q$ with odd $q$, we may write
\begin{align}\label{eq:expansion of chi}
    \chi(b)=\frac{1}{q}\sum_{0<\vert r\vert<q/2} \tau(r,\chi)e\Big(-\frac{br}{q}\Big).
\end{align}
This implies
\begin{align*}
    \sum_{1\le b\le \frac{q-1}{2}}\chi(b)=\frac{1}{q}\sum_{0<\vert r\vert<q/2}\frac{\tau(r,\chi)\delta_{r,q}}{1-e(-r/q)},
\end{align*}
where $\delta_{r,q}=e(-r/q)(1-e(-r(q-1)/2q))$ and $\vert\delta_{r,q}\vert\leq 2$.
For each $n$ with $(n, q)=1$, by (\ref{eq:Gauss sum}) and the orthogonality of Dirichlet characters, we have 
\begin{align}
    \sum_{\chi(-1)=-1}\chi(n)\tau(r,\chi)&=\sum_{\chi(-1)=-1}\chi(n)\sum_{c\bmod q}\chi(c)e\Big(\frac{rc}{q}\Big)\notag\\
    &=\sum_{c\bmod q}e\Big(\frac{rc}{q}\Big)\sum_{\chi(-1)=-1}\chi(nc)\notag\\
    &=\frac{\varphi(q)}{2}\Big(e\Big(\frac{r\overline{n}}{q}\Big)-e\Big(-\frac{r\overline{n}}{q}\Big)\Big).\label{eq:after sum over chi}
\end{align}
Combining (\ref{eq:expansion of chi}) and (\ref{eq:after sum over chi}), we see that (\ref{eq:second term}) equals
\begin{align}\label{eq:sum involve Omega}
    \frac{1}{2q}\sum_{0<\vert r\vert<q/2}\frac{\delta_{r,q}}{1-e(-r/q)}(\Omega_1(r, q; N)-\Omega_2(r, q; N)),
\end{align}
where
\begin{align*}
    &\Omega_1(r, q; N)=2\sum_{n\le N/2}a_{2n}\Big(e\Big(\frac{r\overline{4n}}{q}\Big)-e\Big(-\frac{r\overline{4n}}{q}\Big)\Big)=2\sum_{j=1}^2\sum_{n\le N/2} a_{2n}(-1)^je\Big(\frac{(-1)^j r \overline{4n}}{q}\Big),\\
    &\Omega_2(r, q; N)=\sum_{n\le N}a_n\Big(e\Big(\frac{r\overline{2n}}{q}\Big)-e\Big(-\frac{r\overline{2n}}{q}\Big)\Big)=\sum_{j=1}^2\sum_{n\le N}a_n(-1)^je\Big(\frac{(-1)^j r\overline{2n}}{q}\Big).
\end{align*}
Since $0<\vert r\vert<q/2$, we have 
\begin{align}\label{eq:sum over b estimate}
   \frac{1}{\vert 1-e(-r/q)\vert}=\frac{1}{2|\sin\pi r/q|}\ll\frac{q}{\vert r\vert}.
\end{align}
Hence, (\ref{eq:sum involve Omega}) is at most
\begin{align*}
    \ll \sum_{0<\vert r\vert<q/2}\frac{1}{\vert r\vert}\Big(\vert\Omega_1(r, q; N)\vert+\vert\Omega_2(r, q; N)\vert\Big).
\end{align*}
Thus the lemma follows.
\end{proof}

\begin{remark}
    The incomplete Kloosterman sums appear in our proof because we adopt a different way to treat the RHS of (\ref{eq:restrict chi}). To illustrate the underlying ideas, we simply consider
    \begin{align}\label{eq:remark}
        \cE:=\frac{1}{\varphi(q)}\sum_{\substack{\chi\bmod q\\\chi(-1)=-1}}\Big(\sum_{a\le N}\chi(a)\Big)\Big(\sum_{1\le b\le \frac{q-1}{2}}\chi(b)\Big).
    \end{align}
    Zheng \cite{Zhe93} used the completing method, that is, using the Fourier expansion (\ref{eq:expansion of chi}) of characters to transform the sum over both $a$ and $b$ in (\ref{eq:remark}) to get
    \begin{align}\label{eq:transform in Zheng}
       \cE=\frac{1}{\varphi(q)q^2}\sum_{\substack{\chi\bmod q\\\chi(-1)=-1}}\Big(\sum_{0<\vert m\vert<q/2}\frac{\tau(m,\chi)f_1(m)}{1-e(-m/q)}\Big)\Big(\sum_{0<\vert r\vert<q/2}\frac{\tau(r, \chi)f_2(r)}{ 1-e(-r/q)}\Big),
    \end{align}
where $f_1(m), f_2(r)\ll 1$. The sum over $\chi$ gives
\begin{align*}
    \frac{1}{\varphi(q)}\sum_{\substack{\chi\bmod q\\\chi(-1)=-1}}\tau(r,\chi)\tau(m,\chi)=\frac{1}{2}\Big(S(r, m;q)-S(r, -m; q)\Big).
\end{align*}
By Weil’s bound for Kloosterman sums, together with estimates for geometric sums, the RHS of (\ref{eq:transform in Zheng}) is bounded by $O(q^{1/2+\varepsilon})$. The factor $q^{1/2}$ corresponds to the error term in (\ref{eq:zheng}) and this means one can only obtain the nontrivial estimate in the Pólya–Vinogradov range $N\gg q^{1/2+\varepsilon}$.

In the present paper, our goal is to obtain nontrivial estimates for \eqref{eq:remark} with smaller $N$. It's difficult to explore cancellations in the $a$-sum directly, thus we keep the sum over $a$ in (\ref{eq:remark}) and use (\ref{eq:expansion of chi}) to transform the sum over $b$ as above. Comparing with (\ref{eq:transform in Zheng}), we have
\begin{align*}
    \cE=\frac{1}{\varphi(q)q}\sum_{\substack{\chi\bmod q\\\chi(-1)=-1}}\Big(\sum_{a\le N}\chi(a)\Big)\Big(\sum_{0<\vert r\vert<q/2}\frac{\tau(r, \chi)f_2(r)}{1-e(-r/q)}\Big).
\end{align*}
Now, for $(a,q)=1$, the sum over $\chi$ gives
\begin{align*}
    \frac{1}{\varphi(q)}\sum_{\substack{\chi\bmod q\\\chi(-1)=-1}}\chi(a)\tau(r,\chi)=\frac{1}{2}\Big(e\Big(\frac{r\overline{a}}{q}\Big)-e\Big(-\frac{r\overline{a}}{q}\Big)\Big).
\end{align*}
Hence the sum over $a$ gives the incomplete Kloosterman sums instead of complete ones $S(*,*; q)$. When the modulus $q$ has good factorizations, it is possible to beat the classical Pólya–Vinogradov range as presented in Lemma \ref{lm:incompletekloosterman}.
\end{remark}

\section{Proof of Theorem \ref{thm:main}}\label{section:N small}
To prove Theorem \ref{thm:main}, we apply Lemma \ref{lm:counting-expsums} with $a_n\equiv 1$. Hence, the first term on the RHS of (\ref{eq:transform into character sums})
\begin{align*}
    \sum_{\substack{a\le N\\(a, q)=1}}1=\sum_{d\vert q}\mu(d)\Big[\frac{N}{d}\Big]=N\varphi(q)q^{-1}+O(\tau(q)),
\end{align*}
and the error term essentially is
    \begin{align}\label{eq:bound the error term}
        \sum_{0<\vert r\vert<q/2}\frac{1}{|r|}\Big\vert\sum_{a\in I}e\Big(\frac{ kr\overline{a}}{q}\Big)\Big\vert,
    \end{align}
where $\vert I\vert\asymp N$ and $(k,q)=1$. For the inner sum, we apply Lemma \ref{lm:incompletekloosterman} to get
\begin{align*}
    \sum_{a\in I}e\Big(\frac{kr\overline{a}}{q}\Big)\ll q^{\kappa+\varepsilon/2}N^{\lambda-\kappa+O(\eta)}(r, q),
\end{align*}
where $(\kappa, \lambda)$ is defined as in (\ref{eq:exponent pairs}). Hence, (\ref{eq:bound the error term}) can be bounded by
\begin{align*}
     \ll q^{\kappa+\varepsilon/2}N^{\lambda-\kappa+O(\eta)}\sum_{0<\vert r\vert<q/2}\frac{(r, q)}{\vert r\vert}\ll N^{\lambda-\kappa+O(\eta)}q^{\kappa+\varepsilon}.
\end{align*}
Now Theorem \ref{thm:main} follows.

\section{Counting with structures}

As one may see from above, Lemma \ref{lm:counting-expsums} allows us to study the general counting function
    \begin{align*}
     \sum_{n\in \mathcal{L}(N, q)}a_n,
    \end{align*}
as long as the exponential sum
\begin{align}\label{eq:generalexponentialsum}
\sum_{n\leqslant N}a_ne\Big(\frac{k\overline{n}}{q}\Big)
\end{align}
can be bounded non-trivially in an explicit way.
The difficulty in estimating \eqref{eq:generalexponentialsum} relies on the structure of $\{a_n\}$ and the size of $N$ compared to $q$.

We now take $a_n=\Lambda(n),$ the von Mangoldt function, or the indicator function of primes, then \eqref{eq:generalexponentialsum} reduces to exponential sums in prime variables. For prime $q$, such exponential sums have been studied in depth by Fouvry and Michel \cite{FM98}, Bourgain \cite{Bo05} and Fouvry, Kowalski and Michel \cite{FKM14}. 
It should be worthwhile to mention that \cite{FM98} and \cite{FKM14} deal with sums of $e(f(n)/q)$ with general rational function $f$, or trace functions defined over finite fields, and the arguments use quite a lot of $\ell$-adic cohomology from algebraic geometry. More precisely, their bounds are non-trivial as long as $N>q^{3/4+\varepsilon}.$ On the other hand, with ingenious inputs from additive combinatorics, Bourgain \cite[Theorem A.9]{Bo05} can succeed with non-trivial bounds when $N>q^{1/2+\varepsilon}.$ We recall Bourgain's estimate as follows.

\begin{lemma}[{Bourgain, \cite[Theorem A.9]{Bo05}}]\label{lm:Bourgain}
   Let $q$ be a large prime and $(k,q)=1$. For any fixed $\eta>0$ and $q^{1/2+\eta}<N<q$, there exists a constant $\delta=\delta(\eta)>0$ such that
   \begin{align*}
\sum_{p\leqslant N,\text{ prime}}e\Big(\frac{k\overline{p}}{q}\Big)\ll N^{1-\delta},
\end{align*}
where the implied constant depends at most on $\eta.$
\end{lemma}

Combining Bourgain's estimate with the above arguments, we are able to conclude the following prime-variable version of Lehmer's problem.
To this end, denote by $\cP$ the set of all primes, and by $\pi(N)$ the number of primes up to $N.$ The prime number theorem asserts that $\pi(N)\sim N/\log N.$

\begin{theorem}
     Let $q$ be a large prime. For any fixed $\eta>0$ and $q^{1/2+\eta}< N\le q-1$, there exists some constant $\delta=\delta(\eta)>0$ such that
    \begin{align*}
     |\mathcal{L}(N, q)\cap \cP|=\frac{1}{2}\pi(N)+O(N^{1-\delta}),
        \end{align*}
    where the implied constant depends at most on $\eta$.
\end{theorem}

\section{Acknowledgments}
The author thanks his advisor, Professor Ping Xi, for introducing this problem and useful discussions. This work is supported in part by Shaanxi NSF (No. 2025JC-QYCX-002), Shaanxi Fundamental Science Research Project for Mathematics and Physics (No.25JSZ007) and China Scholarship Council.

\end{document}